\documentclass[a4paper,11pt]{amsart}
\usepackage{./preamble}
\usepackage{braket}

\title[Reduction to t.d.l.c.s.c. groups]{Kirchberg-Wassermann exactness vs exactness: reduction to the unimodular totally disconnected case}

\author{Chris Cave}
\address{Department of Mathematical Sciences, University of Copenhagen,
\newline Universitetsparken 5, DK-2100 Copenhagen \O, Denmark}
\email{chris.cave@math.ku.dk}

\author{Joachim Zacharias}
\address{School of Mathematics and Statistics, University of Glasgow,
\newline  University Place, Glasgow G12 8QQ, Scotland}
\email{joachim.zacharias@glasgow.ac.uk}

\thanks{The first author is supported by the Danish National Research Foundation through the Centre for Symmetry and Deformation (DNRF92).}

\begin{document}

\begin{abstract}
We show that in order to prove that every second countable locally compact groups with exact reduced group C*-algebra is exact in the dynamical sense (i.e. KW-exact) it suffices to show this for totally disconnected groups. 
\end{abstract}

 \maketitle
 
 \section{Introduction}
 There are two natural notions of exactness for locally compact groups which to our knowledge were first mentioned by Kirchberg in  \cite{KirchbergICM94}. A weak one, called here $C^*$-exactness which says that the reduced group algebra is an exact $C^*$-algebra, and a strong one called KW-exactness, which asserts that given any exact sequence of dynamical systems over the group the corresponding sequence of reduced crossed products is exact. The stronger exactness property can thus be  regarded as a dynamical form of exactness. Here  KW stands for Kirchberg and Wassermann who introduced and studied these notions in \cite{Kirchberg1999} and \cite{Kirchberg1999a}.  Since the crossed product by trivial actions is just the tensor product by the reduced group algebra it is evident that KW-exactness implies $C^*$-exactness. As announced in \cite{KirchbergICM94} and later proved in \cite{Kirchberg1999}, the two concepts are equivalent for discrete groups but whether the same equivalence holds true in the case of general second countable locally compact groups has been an open problem ever since. 
 
 Note that there are numerous other concepts related to exactness  such as amenability at infinity or the non-existence of non-compact ghost operators, which have been studied and put forward in the past decades (\cite{Anantharaman-Delaroche2002, Roe2014, Willett2009}). There has been considerable recent progress in the understanding of these conditions showing that  they are equivalent to KW-exactness (\cite{Brodzki2017}). In view of those developments the  question of equivalence of KW-exactness and $C^*$-exactness appears more pressing than ever.
 
 In this note we do not answer this question but reduce the problem to the case when the group is unimodular and totally disconnected.  Thus if $C^*$-exact but non KW-exact groups exist then there must also exist totally disconnected unimodular such groups. This had already been suspected by experts (see the introduction of  \cite{Anantharaman-Delaroche2002}).

 \section*{Acknowledgements}
 We would like to thank Kang Li and Sven Raum for multiple stimulating discussions.
 
 \section{Preliminaries}
 \subsection{$C^*$-exactness and KW-exactness}
 
 As is well-known a $C^*$-algebra $A$  is exact if for any exact sequence  
 \[
   0 \to I  \to E \to Q \to 0
  \]
 of $C^*$-algebras the sequence of minimal tensor products 
  \[
   0 \to A \otimes I  \to A \otimes E \to A \otimes Q \to 0
  \]
is exact. Kirchberg and partly Wassermann proved that this property is equivalent to nuclear embeddability and passes to subalgebras and quotients (c.f. \cite[Chapter 10]{Brown2008}). Exactness of the second sequence can only fail in the middle. That is the kernel of the map onto $A \otimes Q$ is strictly larger than $ A \otimes I$. It is easy to check directly from the definition that a minimal tensor product $A \otimes B$ of two $C^*$-algebras is exact iff $A$ and $B$ are exact (\cite[Proposition 10.2.7]{Brown2008}).
 \begin{defn}
  Let $G$ be a locally compact group. Then $G$ is said to be \emph{$C^*$-exact} if $C^*_r(G)$ is an exact $C^*$-algebra.
 \end{defn}
 If $A$ is a $C^*$-algebra and $G$ is a locally compact group acting on $A$ by $\alpha \colon G \to \mathrm{Aut}(A)$ then the action $\alpha$ is called \emph{continuous} if for all $a \in A$, the map $g \mapsto \alpha_g(a)$ is norm continuous.
 
 \begin{defn}
  Let $G$ be a locally compact group. $G$ is said to be \emph{KW-exact} (KW for Kirchberg and Wassermann) if for all $C^*$-algebras $A$ and all continuous actions $\alpha \colon G \to \mathrm{Aut}(A)$ and for all closed two-sided ideals $I \unlhd A$ such that $\alpha_g(I) = I$ for all $g \in G$, the  sequence
  \[
   0 \to I \rtimes_{\alpha,r} G \to A \rtimes_{\alpha,r} G \to A/I \rtimes_{\alpha,r} G \to 0
  \]
  is exact.
 \end{defn}
 By recent results in \cite{Roe2014, Brodzki2017} it is now known that it suffices to check exactness of only one such sequence, that is, $G$ is KW-exact iff  
 $$
 0 \to C_0(G) \rtimes_{L,r} G \to C_b^{lu} (G) \rtimes_{L,r} G \to  (C_b^{lu} (G)/  C_0(G))\rtimes_{L,r} G \to 0
 $$
 is exact,  where $L$ is the left translation action on the $C^*$-algebra of bounded left uniformly continuous functions $C_b^{lu} (G)$ on $G$.
 
 As already mentioned, since $A \otimes C^*_r(G) \cong A \rtimes_{\tau,r}G$ where $\tau \colon G \to \mathrm{Aut}(A)$ is the trivial action, we have: 
 \begin{prop}
  If $G$ is KW-exact then it is $C^*$-exact.
 \end{prop}
 
 KW-exactness satisfies the following permanence properties.
 \begin{prop}
  Let $G$ be a locally compact group.
  \begin{enumerate}
  \item If $G$ is amenable then $G$ is KW-exact \cite[Proposition 6.1]{Kirchberg1999a}.
  \item If $G$ is connected then $G$ is KW-exact \cite[Theorem 6.8]{Kirchberg1999a}.
    \item Let $N \trianglelefteq G$ be a closed normal subgroup. If $N$ and $G/N$ are KW-exact then $G$ is KW-exact \cite[Theorem 5.1]{Kirchberg1999a}.
 \end{enumerate}
 \end{prop}
Given a subgroup $H$ of a locally compact group $G$, elements in $H$ and $C_r^*(H)$ only act as multipliers on $C_r^*(G)$. However if $H$ is open in $G$ then it is easy to see that $C_r^*(H) \subset C_r^*(G)$. Since exactness passes to subalgebras we get.
 \begin{prop}
If $G$ is a locally compact $C^*$-exact group and $H \leq G$ is an open subgroup then $C^*_r(H) \hookrightarrow C^*_r(G)$ is an injective $*$-homomorphism and so $C^*_r(H)$ is also exact.
\end{prop}

\subsection{Structure of locally compact groups}
The following proposition follows from the closure properties of the class of amenable locally compact groups. We indicate the proof for the reader's convenience.
\begin{prop} [{\cite[Proposition 4.1.12]{Zimmer1984}}]
 Every locally compact group $G$ has a unique maximal amenable closed normal subgroup.
\end{prop}
\begin{proof} Since unions of directed systems of amenable subgroups of $G$ are amenable one only needs to show that given two closed normal amenable subgroups $H_1$ and $H_2$ the closed subgroup $H$ generated by them is amenable. Now the semidirect product $H_1 \rtimes H_2$ is amenable and $H$ is the closure of the continuous image of   $H_1 \rtimes H_2$. This implies that $H$ is also amenable.
\end{proof}

\begin{defn}
 Let $G$ be a locally compact group. Then the \emph{amenable radical}, denoted by $\mathrm{Rad}(G)$ is the unique maximal amenable closed normal subgroup of $G$.
\end{defn}
We have the following characterisation of totally disconnected  locally compact groups which is a classical result by van Danzig.
 \begin{thm}[{\cite{VanDantzig1936}}]
  Let $G$ be a locally compact group. Then $G$ is totally disconnected if and only if it admits a neighbourhood basis of the identity consisting of compact open subgroups.
 \end{thm}
 
 We use the following structure theorem of locally compact groups which is deduced from a solution to Hilbert's fifth problem \cite[Theorem 4.6]{Montgomery1955}. Recall that a subgroup $H \leq G$ is \emph{characteristic} if it is preserved under every automorphism in $\mathrm{Aut}(G)$. 
  
 \begin{thm}[{\cite[Theorem 3.3.3]{Burger2002}, \cite[Theorem 23]{Gheysens2017}}] \label{thm:structure theorem on quotient by amenable radical}
  Let $G$ be any locally compact group. The quotient group $G / \mathrm{Rad}(G)$ has a finite index open characteristic subgroup which splits as a direct product $S \times D$  where $S$ is a connected semi-simple Lie group and $D$ is totally disconnected.
 \end{thm}

 \section{Reduction to the unimodular totally disconnected locally compact second countable case}
 
The aim of this section is to prove the following theorem.

\begin{thm} \label{thm: main result of reduction to tdlc case}
 If KW-exactness and $C^*$-exactness are equivalent for all unimodular totally disconnected second countable groups then they are equivalent for all locally compact second countable groups.
\end{thm}
\subsection{Induced representations and weak containment}
\subsubsection{Induced representations}
Let $G$ be a locally compact group and $H \leq G$ a closed subgroup. For a Borel measure $\nu$ on $G/H$ and $g \in G$, denote $\nu_g$ to be the measure defined as $\nu_g(E) = \nu(gE)$ for all Borel sets $E \subset G/H$. A regular Borel measure $\nu$ is \emph{quasi-invariant} if $\nu_g \sim \nu$ for all $g \in G$ where $\sim$ denotes mutual absolute continuity of measures.

Let $\mu_H$ be a Haar measure on $H$ and define a mapping $T_H \colon C_c(G) \to C_c(G/H)$ where
\[
 T_H(f) (xH) = \int_H f(xh) \, d\mu_H(h).
\]
This map is surjective \cite[Lemma B.1.2]{BHV08}.
\begin{lem}[{\cite[Lemma B.1.3]{BHV08}}]
 Let $\rho \colon G \to \R^{>0}$ be a continuous function on $G$. Then the following are equivalent
 \begin{enumerate}
  \item for all $g \in G$ and $h \in H$ one has
  \[
   \rho(gh) = \frac{\Delta_H(h)}{\Delta_G(h)} \rho(g);
  \]
  \item The functional $\lambda_{\rho} \colon C_c(G/H) \to \C$ defined by
  \[
   \lambda_{\rho} \circ T_H (f) = \int_G f(g) \rho(g) \, d\mu_G
  \]
  is well-defined and positive.
 \end{enumerate}
 If the above conditions hold then the associated regular Borel measure $\mu_{\rho}$ to $\lambda_{\rho}$ under the Riesz representation is quasi-invariant with Radon--Nikodym derivative
 \[
  \frac{d (\mu_{\rho})_y}{d \mu_{\rho}}(xH) = \frac{\rho(yx)}{\rho(x)} \quad \forall x,y \in G.
 \]
\end{lem}

Such a  function $\rho \colon G \to \R^{>0}$ is called a \emph{rho-function for the pair $(G,H)$}. For every pair $(G,H)$ there always exists a rho-function for $(G,H)$.
Indeed if $f \in C_c(G)_+$  then 
\[
\rho_f(x) = \int_H \Delta_G(h) \Delta_H(h)^{-1} f(xh) d\mu_H(h)
\]
is a continuous rho-function.
Thus there always exists a quasi-invariant regular Borel measure on $G / H$. In fact every quasi-invariant regular Borel measure is associated to a rho-function for $(G,H)$ \cite[Theorem B.1.4]{BHV08}. When $H$ is a closed normal subgroup then one takes $\rho = 1$ and the associated quasi-invariant regular Borel measure is the usual Haar measure on $G / H$.

Let $\pi \colon H \to \mcal{U}(\Hs)$ be a unitary representation and fix a quasi-invariant measure $\mu$ on $G /H$. Define a new Hilbert space
\[
 \Hs(\pi) = \Set{\xi \colon G \to \Hs | \xi(xh) = \pi(h^{-1}) \xi(x) \mbox{ and } \int_{G/H} \norm{\xi(x)}^2 \, d\mu(x) < \infty}
\]
with inner product given by
\[
 \Braket{\xi, \eta} = \int_{G/H} \Braket{\xi(x), \eta(x)} \, d\mu(x)
\]
The \emph{representation of $G$ induced from $\pi$} or simply the \emph{induced representation} is the representation $\mathrm{ind}_H^G \pi \colon G \to \mcal{U}(\Hs(\pi))$ given by
\[
 \mathrm{ind}_H^G \pi (x) \xi(y) = \left(\frac{d \mu_{x^{-1}}}{d \mu}(yH)\right)^{1/2} \xi(x^{-1} y) \quad \forall \xi \in \Hs(\pi) \ \forall x,y \in G.
\]
Given another quasi-invariant measure on $G/H$ the same construction gives a unitarily equivalent representation so we can call $\mathrm{ind}_H^G \pi$ the induced representation of $\pi$ to $G$ without trepidation \cite[Proposition E.1.5]{BHV08}.

When $H = \set{e}$ then $\mathrm{ind}_H^G( 1_H) = \lambda_G$. More generally, the representation $\mathrm{ind}_H^G (1_H)$ is called the \emph{quasi-regular representation of $G/H$}. If $H$ is normal in $G$ then $\mathrm{ind}_H^G (1_H)$ is unitarily equivalent to $\lambda_{G/H} \circ q$ where $q \colon G \to G/H$ is the natural surjection and $\lambda_{G/H}$ is the left regular representation of $G/H$ \cite[Proposition 2.38]{Kaniuth2013}.

\subsubsection{Weak containment}
The following definition for weak containment of representations is not standard; however, it is sufficient for our applications.
\begin{defn}[{\cite[Theorem F.4.4]{BHV08}}]
 Let $\pi$ and $\rho$ be unitary representations of $G$. Then $\pi$ \emph{is weakly contained in $\rho$}, denoted by $\pi \prec \rho$, if $\norm{\pi(f)} \leq \norm{\rho(f)}$ for all $f \in L^1(G)$.
\end{defn}

We have the following properties of weak containment and induced representations.
\begin{prop}
 Let $G$ be a locally compact group and $H \leq G$ a closed subgroup. Then
 \begin{enumerate}
  \item $\mathrm{ind}_H^G(\lambda_H)$ is unitarily equivalent to $\lambda_G$ \cite[Corollary 2.52]{Kaniuth2013};
  \item $G$ is amenable if and only if $1_G \prec \lambda_G$ \cite[Theorem G.3.2]{BHV08};
  \item if $\sigma$ and $\rho$ are unitary representations on $H$ and $\sigma \prec \rho$ then $\mathrm{ind}_H^G (\sigma) \prec \mathrm{ind}_H^G (\rho)$ \cite[Theorem F.3.5]{BHV08}.
  \end{enumerate}

\end{prop}



 We believe the following is well known but we provide a proof as we could not find a reference. 
 \begin{lem} \label{lem:closed normal amenable subgroup C star exact}
  Let $G$ be a locally compact second countable group and suppose $H \leq G$ is a closed normal amenable subgroup. If $G$ is $C^*$-exact then $G/H$ is $C^*$-exact.
 \end{lem}
 
 \begin{proof}
  As $H$ is amenable it follows that $1_H \prec \lambda_H$ and so $\mathrm{ind}_H^G(1_H) \prec \lambda_G$. 
  However $\mathrm{ind}_H^G(1_H)$ is unitarily equivalent to $\lambda_{G/H} \circ q$ where $q \colon G \to G/H$ is the quotient map and $\lambda_{G/H}$ is the left regular representation of $G/H$. 
  Hence $\lambda_{G/H} \circ q \prec \lambda_G$ and so it remains to show that the natural map $T_H$ defined by
  \[
 T_H \colon C_c(G) \to C_c(G/H), \quad T_H(f)(gH) = \int_H f(gh) \, d\mu_H(h)
  \]
 extends to a surjective $*$-homomorphism from $C^*_r(G) \to C^*_r(G/H)$. So let $\lambda_{G/H} \circ q \colon C_c(G) \to \mathcal{B}(L^2(G/H))$ be the natural extension of $\lambda_{G/H} \circ q$. That is
  \[
   \lambda_{G/H} \circ q (f) = \int_G f(g) \lambda_{G/H}(gH) \, dg
  \]
  for all $f \in C_c(G)$. We will show that $\lambda_{G/H} \circ q (f) = \lambda_{G/H}(T_H(f))$ for all $f \in C_c(G)$. Then it will follow that
  \[
   \norm{\lambda_{G/H}(T_H(f))} = \norm{\lambda_{G/H} \circ q(f)} \leq \norm{\lambda_G (f)}
  \]
  as $\lambda_{G/H} \circ q \prec \lambda_G$ and so $T_H$ extends to a surjective $*$-homomorphism from $C^*_r(G) \to C^*_r(G/H)$. So for all $f \in C_c(G)$ and $\xi \in C_c(G/H)$ and $yH \in G/H$ we have
  \begin{align*}
   \lambda_{G/H} \circ q (f) \xi (yH) = \int_G f(g)  \xi(g^{-1}yH) \, dg & = \int_{G/H} \int_H f(xh) \xi(h^{-1} x^{-1} y H) \, dh dx \\
   & = \int_{G/H} \int_H f(xh) \xi(x^{-1} y H) \, dh dx
  \end{align*}
where the second equality follows from Weil's integration formula \cite[Corollary 1.21]{Kaniuth2013} and the final equality follows from normality of $H$. Now
\begin{align*}
 \lambda_{G/H}(T_H(f)) \xi (yH) & = \int_{G/H} T_H f(xH) \lambda_{G/H}(xH) \xi(yH) \, d(xH) \\
 & = \int_{G/H} \int_H f(xh) \xi(x^{-1} y H) \, dh dx.
\end{align*}
Hence $\lambda_{G/H} \circ q (f) = \lambda_{G/H} (T_H(f))$ for all $f \in C_c(G)$ and so $T_H$ extends to a surjection. As $C^*_r(G)$ is exact it follows that $C^*_r(G/H)$ is also exact as exactness passes to quotients.
 \end{proof}
 \subsection{Reduction to totally disconnected case}
  
  \begin{lem} \label{lem:unimodular tdlcsc}
  If KW-exactness and $C^*$-exactness are equivalent in the class of unimodular totally disconnected locally compact second countable groups then they are equivalent in the class of totally disconnected locally compact second countable groups
  \end{lem}
  
  \begin{proof}
   Let $G$ be a totally disconnected locally compact second countable group and suppose $G$ is $C^*$-exact. Let $G_0 = \mathrm{ker}(\Delta)$. In particular $G_0$ is a closed normal unimodular subgroup of $G$. As $G$ is totally disconnected, there exists a compact open subgroup $K \leq G$. Compact groups are unimodular it follows that $\Delta|_K = 1$. Hence $K \leq G_0$ and so $\mu_G(G_0) \geq \mu_G(K) > 0$ and so $G_0$ is open. Thus as $G$ is $C^*$-exact it follows that $G_0$ is $C^*$-exact.
   
   By assumption this implies that $G_0$ is KW-exact and as $G / G_0$ is abelian it follows that $G$ is also KW-exact.
  \end{proof}

We are now ready to prove the main result of this section.
\begin{proof}[Proof of Theorem \ref{thm: main result of reduction to tdlc case}]
 Let $G$ be a $C^*$-exact locally compact second countable group. Let $\mathrm{Rad}(G)$ be the amenable radical of $G$. Then by Lemma \ref{lem:closed normal amenable subgroup C star exact} it follows that $G / \mathrm{Rad}(G)$ is $C^*$-exact. By Theorem \ref{thm:structure theorem on quotient by amenable radical} there exists an open normal finite index subgroup $N \leq G / \mathrm{Rad}(G)$ such that $N \cong S \times D$ where $S$ is a connected semisimple Lie group and $D$ is totally disconnected. We have the tensor decomposition where $C^*_r(N) \cong C^*_r(S) \otimes C^*_r(D)$. As $C^*_r(N)$ is exact if follows that $C^*_r(D)$ is also exact \cite[Proposition 10.2.7]{Brown2008}.
 
 By assumption and by Lemma \ref{lem:unimodular tdlcsc} this implies that $D$ is KW-exact. As connected locally compact groups are KW-exact \cite[Theorem 6.8]{Kirchberg1999a} and KW-exactness is preserved under extensions \cite[Theorem 5.1]{Kirchberg1999a} by closed normal subgroups it follows that $N$ is KW-exact. We know $N$ is open so in particular it is closed in $G/ \mathrm{Rad}(G)$. Further $N$ is cocompact in $G / \mathrm{Rad}(G)$ so $G / \mathrm{Rad}(G)$ is KW-exact. As $\mathrm{Rad}(G)$ is a closed normal and amenable subgroup of $G$ and KW-exactness is preserved under extensions it follows that $G$ is KW-exact.
\end{proof}

\bibliographystyle{plain}
\bibliography{bibliography}
\end{document}